\documentclass[12pt,twoside]{amsart}
\usepackage{amsmath, amsthm, amscd, amsfonts, amssymb, graphicx}
\usepackage{enumerate}
\usepackage[colorlinks=true,
linkcolor=blue,
urlcolor=cyan,
citecolor=red]{hyperref}
\usepackage{mathrsfs}
\newtheorem{theorem}{Theorem}

\newtheorem{lemma}{Lemma}[section]
\newtheorem{remark}{Remark}[section]

\newtheorem{corollary}{Corollary}[section]

\newtheorem{proposition}{Proposition}[section]

\numberwithin{equation}{section}

\begin{document}
\title{Some refinements of classical inequalities}
\author{Shigeru Furuichi$^1$ and Hamid Reza Moradi$^2$}
\maketitle

\begin{abstract}
We give some new refinements and reverses Young inequalities in both additive-type and multiplicative-type for two positive numbers/operators. We show our advantages by comparing with known results. A few applications are also given. Some results relevant to the Heron mean are also considered.
\end{abstract}
\pagestyle{myheadings}
\markboth{\centerline {On some refinements for classical inequalities}}
{\centerline {S. Furuichi \& H.R. Moradi}}
\bigskip
\bigskip
\section{Introduction and Preliminaries}
\vskip 0.4 true cm
In this paper, an operator means a bound linear operator on a Hilbert space $\mathcal{H}$. An operator $X$ is said to be positive (denoted by $X\ge 0$) if $\left\langle Xy,y \right\rangle \ge 0$ for all $y\in \mathcal{H}$, and also an operator $X$ is said to be strictly positive (denoted by $X>0$) if $X$ is positive and invertible. For convenience, we often use the following notations:
\[\begin{aligned}
  & A{{!}_{v}}B\equiv {{\left( \left( 1-v \right){{A}^{-1}}+v{{B}^{-1}} \right)}^{-1}},\qquad A{{\sharp}_{v}}B\equiv {{A}^{\frac{1}{2}}}{{\left( {{A}^{-\frac{1}{2}}}B{{A}^{-\frac{1}{2}}} \right)}^{v}}{{A}^{\frac{1}{2}}}, \\ 
 & {{H}_{v}}\left( A,B \right)\equiv \frac{A{{\sharp}_{v}}B+A{{\sharp}_{1-v}}B}{2},\qquad A{{\nabla }_{v}}B\equiv \left( 1-v \right)A+vB, \\ 
\end{aligned}\]
where $A,B$ are strictly positive operators and $0\le v \le 1$. When $v =\frac{1}{2}$, we write $A!B$, $A\sharp B$, $H\left( A,B \right)$ and $A\nabla B$ for brevity, respectively. 

A fundamental inequality between positive real numbers $a,b$ is the Young inequality, which states 
\[{{a}^{1-v}}{{b}^{v}}\le \left( 1-v \right)a+vb\qquad 0\le v\le 1,\]
with equality if and only if $a=b$. If $v =\frac{1}{2}$, we obtain the arithmetic-geometric mean inequality $\sqrt{ab}\le \frac{a+b}{2}$. Recently, a considerable attention is dedicated to the study of Young inequalities and its operator versions \cite{SC,SM}.

It is well-known that, cf. \cite{furuta1}:
\begin{equation}\label{5}
A{{!}_{v }}B\le A{{\sharp}_{v }}B\le A{{\nabla }_{v }}B \qquad 0 \leq v \leq 1,
\end{equation}
where the second inequality in \eqref{5} is known as the operator arithmetic-geometric mean inequality (or the operator Young inequality). 

Based on the refined scalar Young inequality, Kittaneh and Manasrah \cite{k2} obtained that
\begin{equation}\label{k1}
r\left( A+B-2A\sharp B \right)+A{{\sharp}_{v}}B\le A{{\nabla }_{v}}B\le R\left( A+B-2A\sharp B \right)+A{{\sharp}_{v}}B,
\end{equation}
where $r=\min \left\{ v ,1-v  \right\}$ and $R=\max \left\{ v ,1-v  \right\}$.

Zou et al. \cite{21} refined operator Young inequality with the Kantorovich constant $K\left( x \right)\equiv\frac{{{\left( x+1 \right)}^{2}}}{4x}$, $(x>0)$, and proposed the following result:
\begin{equation}\label{zou}
{{K}^{r}}\left( h \right)A{{\sharp}_{v}}B\le A{{\nabla }_{v}}B,
\end{equation}
where $0<\alpha 'I\le A\le \alpha I\le \beta I\le B\le \beta 'I$ or $0<\alpha 'I\le B\le \alpha I\le \beta I\le A\le \beta 'I$, $h=\frac{\beta }{\alpha }$ and $h'=\frac{\beta '}{\alpha '}$. Note also that the inequality \eqref{zou} improves
Furuichi's result from \cite{Furuichi3}, which includes the well known Specht's ratio instead of
Kantorovich constant.

As for the reverse of the operator Young inequality, under the same conditions, Liao et al. \cite{20} gave the following inequality:
\begin{equation}\label{Liao}
A{{\nabla }_{v }}B\le {{K}^{R}}\left( h \right)A{{\sharp}_{v }}B.
\end{equation}
For more related inequalities and applications, see, e.g., \cite{Furuichi1,Furuichi2,SC,SM}. 

\medskip

This paper intends to give some refinements and reverses for the operator Young inequality via Hermite-Hadamard inequality. That is, the following theorem is one of the main results in this paper. 

\begin{theorem}\label{b}
 Let $A,B$ be strictly positive operators such that $0<h'I\le {{A}^{-\frac{1}{2}}}B{{A}^{-\frac{1}{2}}}\le hI\le I$ for some positive scalars $h$ and $h'$. Then for each $0\le v \le 1$,
\begin{equation}\label{9}
{{m}_{v }}\left( h \right)A{{\sharp}_{v }}B\le A{{\nabla }_{v }}B\le {{M}_{v }}\left( h' \right)A{{\sharp}_{v }}B,
\end{equation}
where \[{{m}_{v }}\left( x \right)\equiv1+\frac{2^v v(1-v)(x-1)^2}{(x+1)^{v+1}}
,\]
and
\[{{M}_{v }}\left( x \right)\equiv1+\frac{v(1-v)(x-1)^2}{2x^{v+1}}.\]
\end{theorem}

The proof of Theorem \ref{b} is given in Section \ref{young}, and its advantage for previously known results are also given by Proposition \ref{our_adv_prop} in Section \ref{section3}.

To state our second main result, we recall that the family of Heron mean \cite{01} for two positive numbers $a$ and $b$ is defined as
\[{{F}_{r,v}}\left( a,b \right)\equiv r{{a}^{1-v}}{{b}^{v}}+\left( 1-r \right)\left\{ \left( 1-v \right)a+vb \right\},\qquad\text{ }0\le v\le 1\text{ }and\text{ }r\in \mathbb{R}.\]
More recently the first author \cite{03} showed that if $r\le 1$, then
\begin{equation}\label{1q}
{{\left( \left( 1-v \right){{a}^{-1}}+v{{b}^{-1}} \right)}^{-1}}\le {{F}_{r,v}}\left( a,b \right),\qquad\text{ }0\le v\le 1.
\end{equation}

\begin{theorem}\label{c}
Let $a,b\ge 0$, $r\in \mathbb{R}$ and $0\le v\le 1$. Define
\[{{g}_{r,v}}\left( a,b \right)\equiv v\left( \frac{b-a}{a} \right)\left\{ r{{\left( \frac{a+b}{2a} \right)}^{v-1}}+\left( 1-r \right) \right\}+1,\]
\[{{G}_{r,v}}\left( a,b \right)\equiv \frac{v}{2}\left( \frac{b-a}{a} \right)\left\{ r{{a}^{1-v}}{{b}^{v-1}}+2-r \right\}+1.\]
\begin{itemize}
\item[(1)] If either $a\le b$, $r \geq 0$ or $b\le a$, $r\leq 0$, then
\[{{g}_{r,v}}\left( a,b \right)\le {{F}_{r,v}}\left( a,b \right)\le {{G}_{r,v}}\left( a,b \right).\]
\item[(2)] If either $a\le b$, $r\leq 0$ or $b\le a$, $r\geq 0$, then
\[{{G}_{r,v}}\left( a,b \right)\le {{F}_{r,v}}\left( a,b \right)\le {{g}_{r,v}}\left( a,b \right).\]
\end{itemize}
\end{theorem}
We show the proof of  Theorem \ref{c} along with its advantage by four propositions in Section \ref{section4}.



\section{On Refined Young Inequalities and Reverse Inequalities}\label{young}
\vskip 0.4 true cm
To achieve our results, we need the well-known Hermite-Hadamard inequality which asserts that if $f:\left[ a,b \right]\to \mathbb{R}$ is a convex (concave) function, then the following chain of inequalities hold:
 \begin{equation}\label{1}
 f\left( \frac{a+b}{2} \right)\le \left( \ge  \right)\frac{1}{b-a}\int\limits_{a}^{b}{f\left( x \right)dx}\le \left( \ge  \right)\frac{f\left( a \right)+f\left( b \right)}{2}.
 \end{equation}

\medskip

Our first attempt, which is a direct consequence of \cite[Theorem 1]{MFM2017}, gives an additive-type improvement and reverse for the operator Young inequality via \eqref{1}. 

To obtain inequalities for bounded self-adjoint operators on Hilbert space, we shall use the following monotonicity property for operator functions: If $X\in \mathcal{B}\left( \mathcal{H} \right)$ is a self-adjoint operator with a spectrum $Sp\left( X \right)$ and $f,g$ are continuous real-valued functions on $Sp\left( X \right)$, then
	\[f\left( t \right)\le g\left( t \right),\text{ }t\in Sp\left( X \right)\qquad \Rightarrow \qquad f\left( X \right)\le g\left( X \right).\] 
\vskip 0.3 true cm
The next lemma provides a technical result which we will need in the sequel.
\begin{lemma}\label{2}
Let $0< v \le 1$.
\begin{itemize}
\item[(i)] For each $t>0$, the function ${{f}_{v }}\left( t \right)=v(1- {{t}^{v -1}})$ is concave.
\item[(ii)] The function ${{g}_{v }}\left( t \right)=\frac{v(1-v)(t-1)}{t^{v+1}}$, is concave if $t \leq 1+\frac{2}{v}$, and convex if $t \geq 1+\frac{2}{v}$.
\end{itemize}
\end{lemma}
\begin{proof}
The function ${{f}_{v }}\left( t \right)$ is twice differentiable and  ${{f}_{v }}''\left( t \right)=v \left(1- v \right)\left( v -2 \right){{t}^{v -3}}$. According to the assumptions $t>0$ and $0\le v \le 1$, so ${{f}_{v }}''\left( t \right)\le 0$.

The function ${{g}_{v }}\left( t \right)$ is also twice differentiable and ${{g}_{v }}''\left( t \right)=v \left( 1- v \right)\left( v +1 \right)\left( \frac{v t-v -2}{{{t}^{v +3}}} \right)$ which implies (ii).
\end{proof}
\vskip 0.3 true cm
Using this lemma, together with \eqref{1}, we have the following proposition.

\begin{proposition}\label{tha}
Let $A,B$ be strictly positive operators such that $A\le B$. Then for each $0\le v \le 1$,
\begin{equation}\label{104}
\begin{aligned}
  & v \left( B-A \right){{A}^{-1}}\left( \frac{A-A{{\natural}_{v -1}}B}{2} \right)+A{{\sharp}_{v}}B \\ 
 & \le A{{\nabla }_{v }}B \\ 
 & \le v \left( B-A \right){{A}^{-1}}\left( A-{{A}^{\frac{1}{2}}}{{\left( \frac{I+{{A}^{-\frac{1}{2}}}B{{A}^{-\frac{1}{2}}}}{2} \right)}^{v -1}}{{A}^{\frac{1}{2}}} \right)+A{{\sharp}_{v }}B. \\ 
\end{aligned}
\end{equation}
\end{proposition}

\medskip

\begin{proof}
In order to prove \eqref{104}, we firstly prove the corresponding scalar inequalities. As we showed in Lemma \ref{2}(i), the function ${{f}_{v }}\left( t \right)=v(1-  {{t}^{v -1}})$ where $t\ge 1$ and $0\le v \le 1$ is concave. Moreover, we readily check that 
	\[\int\limits_{1}^{x}{{{f}_{v }}}\left( t \right)dt=\left( 1-v  \right)+v x-{{x}^{v }}.\] 
From the  inequality \eqref{1} for concave function we infer that
\begin{equation}\label{3}
v \left( x-1 \right)\left( \frac{1-{{x}^{v -1}}}{2} \right)+{{x}^{v }}\le \left( 1-v  \right)+v x\le v \left( x-1 \right)\left( 1-{{\left( \frac{1+x}{2} \right)}^{v -1}} \right)+{{x}^{v }},
\end{equation}
 where $x\ge 1$ and $0\le v \le 1$.

With $X={{A}^{-\frac{1}{2}}}B{{A}^{-\frac{1}{2}}}$ and thus $Sp\left( X \right)\subseteq \left( 1,+\infty  \right)$, relation \eqref{3} holds for any $x\in Sp\left( X \right)$.  Therefore
\[\begin{aligned}
   v \left( X-I \right)\left( \frac{I-{{X}^{v -1}}}{2} \right)+{{X}^{v }}&\le \left( 1-v  \right)I+v X \\ 
 & \le v \left( X-I \right)\left( I-{{\left( \frac{I+X}{2} \right)}^{v -1}} \right)+{{X}^{v }}.  
\end{aligned}\]
Finally, multiplying both sides by ${{A}^{\frac{1}{2}}}$, we get \eqref{104}.
\end{proof}

\medskip

By virtue of Proposition \ref{tha}, we can improve the first inequality in \eqref{5}.
\begin{remark}\label{12}
It is worth remarking that the left-hand side of inequality \eqref{104},  is a refinement of operator Young inequality in the sense of $v \left( x-1 \right)\left( \frac{1-{{x}^{v -1}}}{2} \right)\ge 0$ for each $x\ge 1$ and $0\le v \le 1$, i.e.,
\begin{equation}\label{hrmonic}
\begin{aligned}
  & A{{\sharp}_{v}}B \\ 
 & \le v\left( B-A \right){{A}^{-1}}\left( \frac{A-A{{\natural}_{v-1}}B}{2} \right)+A{{\sharp}_{v}}B \\ 
 & \le A{{\nabla }_{v}}B. \\ 
\end{aligned}
\end{equation}
Replacing $A$ and $B$ by ${{A}^{-1}}$ and ${{B}^{-1}}$ respectively in \eqref{hrmonic}, we obtain
\begin{equation}\label{011}
\begin{aligned}
  & {{A}^{-1}}{{\sharp}_{v}}{{B}^{-1}} \\ 
 & \le v\left( {{B}^{-1}}-{{A}^{-1}} \right)A\left( \frac{{{A}^{-1}}-{{A}^{-1}}{{\natural}_{v-1}}{{B}^{-1}}}{2} \right)+{{A}^{-1}}{{\sharp}_{v}}{{B}^{-1}} \\ 
 & \le {{A}^{-1}}{{\nabla }_{v}}{{B}^{-1}}. \\ 
\end{aligned}
\end{equation}
Taking inverse in \eqref{011}, we get
\[\begin{aligned}
  & A{{!}_{v }}B \\ 
 & \le {{\left\{ v \left( {{B}^{-1}}-{{A}^{-1}} \right)A\left( \frac{{{A}^{-1}}-{{A}^{-1}}{{\natural}_{v -1}}{{B}^{-1}}}{2} \right)+{{A}^{-1}}{{\sharp}_{v }}{{B}^{-1}} \right\}}^{-1}} \\ 
 & \le A{{\sharp}_{v }}B. \\ 
\end{aligned}\]
\end{remark}

\medskip

In order to give a proof of our first main result, we need the following essential result.  
\begin{proposition}\label{6}
For each $0<x\le 1$ and $0\le v \le 1$, the functions
$m_v(x)$ and $M_v(x)$ defined in Theorem \ref{b} are
 decreasing. Moreover $1\le {{m}_{v}}\left( x \right)\le {{M}_{v}}\left( x \right)$.
\end{proposition}
\begin{proof}
The function ${{m}_{v }}\left( x \right)$ is differentiable and 
\[{{m}_{v }}'\left( x \right)=\frac{v \left( v -1 \right){{2}^{v}}}{{{\left( x+1 \right)}^{v +2}}}\left( \left( v -1 \right){{x}^{2}}+v +3-2\left( v +1 \right)x \right).\]
By assumptions we can find easily that ${{m}_{v }}'\left( x \right)\le 0$, for any $0<x\le 1$ and $0\le v \le 1$. In addition ${{m}_{v }}\left( 1 \right)=1$, so ${{m}_{v }}\left( x \right)\ge 1$.

Similarly the function ${{M}_{v }}\left( x \right)$ is differentiable and 
\[{{M}_{v }}'\left( x \right)=\frac{v \left( v -1 \right)\left( x-1 \right)\left( \left( v -1 \right)x-v -1 \right)}{2{{x}^{v +2}}}.\]
Therefore ${{M}_{v }}'\left( x \right)\le 0$ for any $0<x\le 1$ and $0\le v \le 1$. We also have ${{M}_{v }}\left( 1 \right)=1$, i.e., ${{M}_{v }}\left( x \right)\ge 1$. It remains to prove ${{m}_{v }}\left( x \right)\le {{M}_{v }}\left( x \right)$. Suppose that
\[{{\mathfrak{M}}_{v }}\left( x \right)\equiv{{M}_{v }}\left( x \right)-{{m}_{v }}\left( x \right)\qquad 0<x\le 1\text{ and }0\le v \le 1.\]
In a way similar to what we have done above, we can calculate ${{\mathfrak{M}}'_{v }}(x)$ in the following:
$$
{{\mathfrak{M}}'_{v }}(x) = \frac{v(1-v)(1-x)}{2(x+1)^2x^{v+2}} \mathfrak{h}_v(x),
$$
where
$$
\mathfrak{h}_v(x) \equiv 2x^2\left\{ (1-v)x +v+3 \right\} \left(\frac{2x}{x+1}\right)^v -\left\{ (1-v)x^3+(3-v)x^2+(v+3)x+(v+1)\right\}.
$$
Since $0< x \leq 1$, $\left( \frac{2x}{x+1} \right)^{v} \leq 1$. Thus ${{\mathfrak{M}}'_{v }}(x)$ is bounded from the above:
$$
{{\mathfrak{M}}'_{v }}(x) \leq \frac{v(1-v)(1-x)}{2(x+1)^2x^{v+2}} \mathfrak{k}_v(x),
$$
where
$$
\mathfrak{k}_v(x) \equiv (1-v)x^3 + 3(v+1)x^2-(v+3)x-(v+1).
$$
By elementary calculations, we find that
$$
\mathfrak{k}_v''(x) = 6(1-v)x+6(v+1)\geq 0,\qquad \mathfrak{k}_v(0)=-(v+1) <0, \qquad \mathfrak{k}_v(1)=0.
$$
Thus we have $\mathfrak{k}_v(x) \leq 0$ which implies ${{\mathfrak{M}}'_{v }}(x) \leq 0$ so that $ {{\mathfrak{M}}_{v }}\left( x \right) \geq {{\mathfrak{M}}_{v }}\left( 1 \right)=0$. Therefore, the proposition follows.
\end{proof}

\medskip

We are now in a position to prove Theorem \ref{b} which is a multiplicative type refinement and reverse for the operator Young inequality.

\medskip

\noindent{\it Proof of Theorem \ref{b}.} \,\, It is routine to check that the function ${{f}_{v }}\left( t \right)=\frac{v(1-v)(t-1)}{t^{v +1}}$ where $0<t\le 1$ and $0\le v \le 1$, is concave. We can verify that
\[\int\limits_{x}^{1}{{{f}_{v }}\left( t \right)dt}=1-\frac{\left( 1-v  \right)+v x}{{{x}^{v }}}.\]
Hence from the inequality \eqref{1} we can write
\begin{equation}\label{7}
{{m}_{v }}\left( x \right){{x}^{v }}\le \left( 1-v  \right)+v x\le {{M}_{v }}\left( x \right){{x}^{v }},
\end{equation}
for each $0<x\le 1$ and $0\le v\le 1$.   

Now, we shall use the same procedure as in {{\cite[Theorem 2]{Furuichi3}}}. The inequality \eqref{7} implies that
\[\underset{h'\le x\le h \le1}{\mathop{\min }}\,{{m}_{v }}\left( x \right){{x}^{v }}\le \left( 1-v  \right)+v x\le \underset{h'\le x\le h\le1}{\mathop{\max }}\,{{M}_{v }}\left( x \right){{x}^{v }}.\]
Based on this inequality, one can easily see for which $X$ 
\begin{equation}\label{8}
\underset{h'\le x\le h \le1}{\mathop{\min }}\,{{m}_{v }}\left( x \right){{X}^{v }}\le \left( 1-v  \right)I+v X\le \underset{h'\le x\le h \le1}{\mathop{\max }}\,{{M}_{v }}\left( x \right){{X}^{v }}.
\end{equation}
By substituting ${{A}^{-\frac{1}{2}}}B{{A}^{-\frac{1}{2}}}$ for $X$ and taking into account that ${{m}_{v }}\left( x \right)$ and ${{M}_{v }}\left( x \right)$ are decreasing, the relation \eqref{8} implies
\begin{equation}\label{mv}
{{m}_{v}}\left( h \right){{\left( {{A}^{-\frac{1}{2}}}B{{A}^{-\frac{1}{2}}} \right)}^{v}}\le \left( 1-v \right)I+v{{A}^{-\frac{1}{2}}}B{{A}^{-\frac{1}{2}}}\le {{M}_{v}}\left( h' \right){{\left( {{A}^{-\frac{1}{2}}}B{{A}^{-\frac{1}{2}}} \right)}^{v}}.
\end{equation}
Multiplying ${{A}^{\frac{1}{2}}}$ from the both sides to the inequality \eqref{mv}, we have the inequality \eqref{9}.

\hfill \qed

\begin{remark}\label{rem}
Notice that, the condition $0<h'I\le {{A}^{-\frac{1}{2}}}B{{A}^{-\frac{1}{2}}}\le hI\le I$ in Theorem \ref{b}, can be replaced by $0<\alpha 'I\le B\le \alpha I\le \beta I\le A\le \beta 'I$. In this case we have
	\[{{m}_{v }}\left( h \right)A{{\sharp}_{v }}B\le A{{\nabla }_{v }}B\le {{M}_{v }}\left( h' \right)A{{\sharp}_{v }}B,\] 
where $h=\frac{\alpha }{\beta }$ and $h'=\frac{\alpha '}{\beta '}$.  
\end{remark}

\medskip

It is well-known that for each strictly positive operators $A$ and $B$ (see e.g., \cite[Proposition 3.3.11]{Hiai2010}), 
\begin{equation}\label{heinz}
{{H}_{v}}\left( A,B \right)\le A\nabla B\qquad 0\le v\le 1.
\end{equation}
A counterpart to the inequality \eqref{heinz} is as follows:
\begin{remark}
Assume the conditions of Theorem \ref{b}. Then
\[A\nabla B\le \sqrt{{{M}_{v}}{{\left( h'^2\right)}}}{{H}_{v}}\left( A,B \right).\]
\end{remark}

\medskip

Theorem \ref{b} can be used to infer the following remark:
\begin{remark}
Assume the conditions of Theorem \ref{b}. Then
\[{{m}_{v}}\left( h \right)A{{!}_{v}}B\le A{{\sharp}_{v}}B\le {{M}_{v}}\left( h' \right)A{{!}_{v}}B.\]
\end{remark}

\medskip

The left-hand side of inequality \eqref{9} can be squared by a similar method as in \cite{lin1,lin2}.
\begin{corollary} 
Let $0<\alpha 'I\le B\le \alpha I\le \beta I\le A\le \beta 'I$. Then for every normalized positive linear map $\Phi $,
\begin{equation}\label{a4}
{{\Phi }^{2}}\left( A{{\nabla }_{v}}B \right)\le {{\left( \frac{K\left( h' \right)}{{{m}_{v}}\left( h \right)} \right)}^{2}}{{\Phi }^{2}}\left( A{{\sharp}_{v}}B \right)
\end{equation}
and
\begin{equation}\label{a5}
{{\Phi }^{2}}\left( A{{\nabla }_{v}}B \right)\le {{\left( \frac{K\left( h' \right)}{{{m}_{v}}\left( h \right)} \right)}^{2}}{{\left( \Phi \left( A \right){{\sharp}_{v}}\Phi \left( B \right) \right)}^{2}}
\end{equation}
where $h=\frac{\alpha }{\beta }$ and $h'=\frac{\alpha '}{\beta '}$.
\end{corollary}
\begin{proof}
According to the assumptions
\[\left( \alpha '+\beta ' \right)I\ge \alpha '\beta '{{A}^{-1}}+A,\qquad \left( \alpha '+\beta ' \right)I\ge \alpha '\beta '{{B}^{-1}}+B,\]
since $(t-\alpha')(t-\beta') \leq 0$ for $\alpha' \leq t \leq \beta' $. From these we can write
\begin{equation}\label{a2}
\left( \alpha '+\beta ' \right)I\ge \alpha '\beta '\Phi \left( {{A}^{-1}}{{\nabla }_{v}}{{B}^{-1}} \right)+\Phi \left( A{{\nabla }_{v}}B \right),
\end{equation}
where $\Phi $ is a normalized positive linear map. We have
\[\begin{aligned}
  & \left\| \Phi \left( A{{\nabla }_{v}}B \right)\alpha '\beta '{{m}_{v}}\left( h \right){{\Phi }^{-1}}\left( A{{\sharp}_{v}}B \right) \right\| \\ 
 & \le \frac{1}{4}{{\left\| \Phi \left( A{{\nabla }_{v}}B \right)+\alpha '\beta '{{m}_{v}}\left( h \right){{\Phi }^{-1}}\left( A{{\sharp}_{v}}B \right) \right\|}^{2}} \quad \text{(by \cite{bhatia})}\\ 
 & \le \frac{1}{4}{{\left\| \Phi \left( A{{\nabla }_{v}}B \right)+\alpha '\beta '{{m}_{v}}\left( h \right)\Phi \left( {{A}^{-1}}{{\sharp}_{v}}{{B}^{-1}} \right) \right\|}^{2}} \quad \text{(by Choi's
 inequality {{\cite[p. 41]{bhatia1}}})}\\ 
 & \le \frac{1}{4}{{\left\| \Phi \left( A{{\nabla }_{v}}B \right)+\alpha '\beta '\Phi \left( {{A}^{-1}}{{\nabla }_{v}}{{B}^{-1}} \right) \right\|}^{2}} \quad \text{(by Remark \ref{rem})}\\ 
 & \le \frac{1}{4}{{\left( \alpha '+\beta ' \right)}^{2}} \quad \text{(by \eqref{a2})}.\\ 
\end{aligned}\]
This is the same as saying
\begin{equation}\label{a3}
\left\| \Phi \left( A{{\nabla }_{v}}B \right){{\Phi }^{-1}}\left( A{{\sharp}_{v}}B \right) \right\|\le \frac{K\left( h' \right)}{{{m}_{v}}\left( h \right)},
\end{equation}
where $h=\frac{\alpha }{\beta }$ and $h'=\frac{\alpha '}{\beta '}$. It is not hard to see that \eqref{a3} is equivalent to \eqref{a4}. The proof of the inequality \eqref{a5} goes likewise and we omit the details.
\end{proof}
\begin{remark}
Obviously, the bounds in \eqref{a4} and \eqref{a5} are tighter than those in  {{\cite[Theorem 2.1]{lin2}}}, under the conditions $0<\alpha 'I\le B\le \alpha I\le \beta I\le A\le \beta 'I$ with $h=\frac{\alpha }{\beta }$ and $h'=\frac{\alpha '}{\beta '}$.
\end{remark}
\section{Connection With Known Results}\label{section3}
\vskip 0.4 true cm
In this section, we point out connections between our results given in Section \ref{young} and some inequalities
proved in other contexts. That is, we are now going to explain the advantages of our results.
Let $0\le v \le 1$, $r=\min \left\{ v ,1-v  \right\}$, $R=\max \left\{ v ,1-v  \right\}$ and ${{m}_{v}}\left( \cdot \right)$, ${{M}_{v}}\left( \cdot \right)$ were defined as in Theorem \ref{b}.
As we will show in Appendix A,  the following proposition explains the advantages of our results.
\begin{proposition}\label{our_adv_prop}
The following statements are true.
\begin{itemize}
\item[(I-i)] The lower bound of Proposition \ref{tha} improves the  first inequality in \eqref{k1}, when $\frac{3}{4} \leq v \leq 1$ with $0< A \leq B$. 
\item[(I-ii)] The upper bound of Proposition \ref{tha} improves the  second inequality in \eqref{k1}, when $\frac{2}{3} \leq v \leq 1$  with $0< A \leq B$.
\item[(I-iii)] The upper bound of Proposition \ref{tha} improves the  second inequality in \eqref{k1}, when $0 \leq v \leq \frac{1}{3}$  with $0< A \leq B$.
\item[(II)] The upper bound of Theorem \ref{b} improves the inequality
$$
(1-v)+v x \leq x^v K(x),
$$
when $x^v \geq \frac{1}{2}$.
\item[(III)]
The upper bound of Theorem \ref{b} improves the inequality given by  Dragomir in \cite[Theorem 1]{dragomir},
\begin{equation} \label{ineq_Dragomir_comparison}
\left( 1-v \right)+vx\le \exp \left( 4v\left( 1-v \right)\left( K\left( x \right)-1 \right) \right){{x}^{v}}, \quad x >0
\end{equation}
when $0 \leq v \leq \frac{1}{2}$ and $0< x \leq 1$. 
\item[(IV)] There is no ordering between  Theorem \ref{b} and the inequalities
 (\ref{zou}) and (\ref{Liao}). 
\end{itemize}
\end{proposition}
Therefore we conclude that Proposition \ref{tha}  and Theorem \ref{b} are not trivial results. The proofs in the above mentioned are given in Appendix A.

\section{Inequalities Related to Heron Mean}\label{section4}
\vskip 0.4 true cm

This section aims to prove new inequalities containing \eqref{1q}. These inequalities were given in Theorem \ref{c}. Our main idea and technical tool are closely related to the inequalities \eqref{1}. 

\medskip

\noindent{\it Proof of Theorem \ref{c}.}\,\, Consider the function ${{f}_{r,v}}\left( t \right)\equiv rv{{t}^{v-1}}+\left( 1-r \right)v$ where $t>0$, $r\in \mathbb{R}$ and $0\le v\le 1$. Since the function ${{f}_{r,v}}\left( t \right)$ is twice differentiable, one can easily see that 
	\[\frac{d{{f}_{r,v}}\left( t \right)}{dt}=r\left( v-1 \right)v{{t}^{v-2}},\qquad \frac{{{d}^{2}}{{f}_{r,v}}\left( t \right)}{d{{t}^{2}}}=r\left( v-2 \right)\left( v-1 \right)v{{t}^{v-3}}.\] 
 It is not hard to check that 
\[\left\{ \begin{aligned}
  & \frac{{{d}^{2}}{{f}_{r,v}}\left( t \right)}{d{{t}^{2}}}\ge 0\quad \text{ for }r\geq 0 \\ 
 & \frac{{{d}^{2}}{{f}_{r,v}}\left( t \right)}{d{{t}^{2}}}\le 0\quad \text{ for }r\leq 0 \\ 
\end{aligned} \right..\]

Utilizing the inequality \eqref{1} for the function ${{f}_{r,v}}\left( t \right)$ we infer that
\begin{equation}\label{01}
{{g}_{r,v}}\left( x \right)\le r{{x}^{v}}+\left( 1-r \right)\left( \left( 1-v \right)+vx \right)\le {{G}_{r,v}}\left( x \right),
\end{equation}
where
\begin{eqnarray} 
&& {{g}_{r,v}} \left( x \right)\equiv v\left( x-1 \right)\left\{ r{{\left( \frac{1+x}{2} \right)}^{v-1}}+\left( 1-r \right) \right\}+1, \label{grv_def01} \\ 
 && {{G}_{r,v}}\left( x \right)\equiv \frac{v\left( x-1 \right)}{2}\left( r{{x}^{v-1}}+ 2-r \right)+1,  \label{grv_def02}
\end{eqnarray}
for each $x\ge 1$, $r\geq 0$ and $0\le v\le 1$. Similarly for each $0< x\le 1$, $r\geq 0$ and $0\le v\le 1$, we get
\begin{equation}\label{02}
{{G}_{r,v}}\left( x \right)\le r{{x}^{v}}+\left( 1-r \right)\left( \left( 1-v \right)+vx \right)\le {{g}_{r,v}}\left( x \right).
\end{equation}
If $x\ge 1$ and $r\leq 0$, we get
\begin{equation}\label{003}
{{G}_{r,v}}\left( x \right)\le r{{x}^{v}}+\left( 1-r \right)\left( \left( 1-v \right)+vt \right)\le {{g}_{r,v}}\left( x \right),
\end{equation}
for each $0\le v\le 1$. For the case $0< x\le 1$ and $r\leq 0$ we have 
\begin{equation}\label{004}
{{g}_{r,v}}\left( t \right)\le r{{x}^{v}}+\left( 1-r \right)\left( \left( 1-v \right)+vt \right)\le {{G}_{r,v}}\left( x \right),
\end{equation}
for each $0\le v\le 1$.

\hfill \qed

\medskip

Note that we equivalently obtain the operator inequalities from the scalar inequalities given in Theorem \ref{c}. We here omit such expressions  for simplicity.

Closing this section, we prove the ordering $\left\{ (1-v)+vt^{-1}\right\}^{-1} \leq g_{r,v}(t) $ and $\left\{ (1-v)+vt^{-1}\right\}^{-1} \leq G_{r,v}(t)$ under some assumptions, for the purpose to show the advantages of our lower bounds given in Theorem \ref{c}. 
It is known that
\[{{\left\{ \left( 1-v \right)+v{{t}^{-1}} \right\}}^{-1}}\le {{t}^{v}}\qquad 0\le v\le 1\text{ and }t>0,\]
so that we also have interests in the ordering $g_{r,v}(t)$ and $G_{r,v}(t)$ with $t^v$. That is, we can show the following four propositions. The proofs are given in Appendix B.

\begin{proposition} \label{prop01_0330}
For $t \geq 1$ and $0 \leq v,r \leq 1$, we have
\begin{equation} \label{ineq01_0330}
\left\{ (1-v)+vt^{-1}\right\}^{-1} \leq g_{r,v}(t).
\end{equation}
\end{proposition}

\begin{proposition} \label{prop02_0330}
For $0 < t \leq 1$ and $0 \leq v,r \leq 1$, we have
\begin{equation} \label{ineq03_0330}
\left\{ (1-v)+vt^{-1}\right\}^{-1} \leq t^v \leq g_{r,v}(t).
\end{equation}
\end{proposition}

\begin{proposition} \label{prop05_0330}
For $0 \leq r,v \leq 1$ and $c \leq t \leq 1$ with $c \equiv \frac{2^7-1}{5^4}$, we have
\begin{equation} \label{ineq07_0330}
\left\{ (1-v)+vt^{-1}\right\}^{-1} \leq G_{r,v}(t).
\end{equation}
\end{proposition}

\begin{proposition} \label{prop06_0330}
For $0 \leq v \leq 1$, $r \leq 1$ and $t \geq 1$, we have
\begin{equation} \label{ineq08_0330}
\left\{ (1-v)+vt^{-1}\right\}^{-1} \leq t^v \leq G_{r,v}(t).
\end{equation}
\end{proposition}

\begin{remark}
Propositions \ref{prop01_0330}-\ref{prop06_0330} show that lower bounds given in Theorem \ref{c} are tighter than the known bound (Harmonic mean), for the cases given in  Propositions \ref{prop01_0330}-\ref{prop06_0330}. 
If $r=1$ in Proposition \ref{prop01_0330}, then $g_{r,v}(t) \leq t^v$, for $t \geq 1$ and $0 \leq v \leq 1$. 
If $r=1$ in Proposition \ref{prop05_0330}, then $G_{r,v}(t) \leq t^v$, for $c \leq t \leq 1$ and $0 \leq v \leq 1$. 
We thus find that
Proposition \ref{prop01_0330} and Proposition \ref{prop05_0330} make  sense for the purpose of finding the functions between $\left\{ (1-v)+vt^{-1}\right\}^{-1} $ and $t^v$.
\end{remark}
\begin{remark}
In the process of the proof in Proposition \ref{prop05_0330} we find the inequality:
$$
\frac{t^v+t}{2} \leq \left\{ (1-v)+vt^{-1}\right\}^{-1} ,
$$
for $0 \leq v \leq 1$ and $c \leq t \leq 1$.
Then we have the following inequalities:
$$
\frac{A \sharp_v B +B}{2} \leq A!_v B \leq A \sharp_v B,
$$
for $0< c A \leq B \leq A$ with $c = \frac{2^7 -1}{5^4}$, and  $0\leq v \leq 1$.

In the process of the proof in Proposition \ref{prop02_0330} we also find the inequality:
$$
t \left( \frac{t+1}{2} \right)^{v-1} \leq \left\{ (1-v)+vt^{-1}\right\}^{-1}, 
$$
for $0 \leq v \leq 1$ and $0 \leq t \leq 1$.
Then we have the following inequalities:
$$
BA^{-1/2} \left(\frac{A^{-1/2}BA^{-1/2} +I}{2} \right)^{v-1} A^{1/2} \leq  A!_v B \leq A \sharp_v B,
$$
for $0<  B \leq A$ and $0\leq v \leq 1$.
\end{remark}
\section*{Concluding Remark}
\vskip 0.4 true cm
Several refinements and generalizations of the inequality \eqref{1} have been given (see, e.g. \cite{1,3,2,4}). Of course, if we apply them with similar considerations were discussed above, we can find new results concerning mean inequalities. We leave the details of this idea to the interested reader, as it is just an application of our main results.

\vskip 0.4 true cm
\noindent{\bf Acknowledgment.} The authors thank anonymous referees for giving valuable comments and suggestions to improve our manuscript. The author (S.F.) was partially supported by JSPS KAKENHI Grant Number
16K05257 

\vskip 0.4 true cm

\bibliographystyle{alpha}

\vskip 0.4 true cm

\section*{Appendix A}
\vskip 0.4 true cm
For the purpose to give proof of Proposition \ref{our_adv_prop},  we need the following lemma.
\begin{lemma} \label{lemma_0326}
For each $x \geq 1$, we have
\begin{equation} \label{ineq04_0326}
\left( \frac{x+1}{2} \right)^{2/3} \geq \left( \frac{\sqrt{x} +1}{2 \sqrt{x}}\right) \left(1+\log\left( \frac{x+1}{2}\right) \right).
\end{equation}
\end{lemma}
\begin{proof}
We firstly prove
\begin{equation} \label{ineq05_0326}
\left( \frac{x+1}{2} \right)^{2/3} \geq \left( \frac{1}{2} +\frac{x+1}{4x}\right) \left(1+\log\left( \frac{x+1}{2}\right) \right),
\end{equation}
for $x \geq 1$. Putting $t = \frac{x+1}{2} \geq 1$, the inequality (\ref{ineq05_0326}) is equivalent to the inequality
$$
t^{2/3} \geq \frac{(3t-1)}{2(2t-1)} (1+ \log t),
$$
which is equivalent to saying
$$
2s^2(2s^3-1) \geq (3s^3 -1) (1+3 \log s),
$$
where $s =t^{1/3} \geq 1$. To prove the above inequality, we set
$$
\mathfrak{F}(s) \equiv 4s^5-3s^3-2s^2+1-9s^3\log s +3 \log s\qquad s \geq 1.
$$ 
By simple calculations, we have
 $\mathfrak{F}(s) \geq \mathfrak{F}(1) =0$. Hence we have the inequality (\ref{ineq05_0326}).
For any $a>0$, we have $\frac{2a}{1+a} \leq \sqrt{a}$, that is, $\frac{a +1}{2a} \geq \frac{1}{\sqrt{a}}$. Therefore for any $a > 0$, we have $\frac{1}{2}+\frac{a+1}{4a}  \geq \frac{1}{2} +\frac{1}{2\sqrt{a}} = \frac{\sqrt{a} +1 }{2\sqrt{a}}$, which implies the following second inequality
$$
\left( \frac{x+1}{2} \right)^{2/3} \geq \left( \frac{1}{2} +\frac{x+1}{4x}\right) \left(1+\log\left( \frac{x+1}{2}\right) \right)\geq \left( \frac{\sqrt{x} +1}{2 \sqrt{x}}\right) \left(1+\log\left( \frac{x+1}{2}\right) \right).
$$
This completes the proof.
\end{proof}

\medskip

\noindent{\it Proof of Proposition \ref{our_adv_prop}.}
\begin{itemize}
\item[(I)]  Assume that $x\ge 1$.
\begin{itemize}
\item[(i)] Consider the function 
\[{{u}_{v }}\left( x \right)\equiv v \left( x-1 \right)\left( \frac{1-{{x}^{v -1}}}{2} \right)-r{{\left( 1-\sqrt{x} \right)}^{2}}.\]
For $\frac{3}{4} \leq v \leq 1$, we have ${{u}_{v}}\left( x \right)\ge 0$. Let us prove this statement. Since $u_1(x) =0$ and
$
\frac{d^2u_v(x)}{dv^2} = \frac{1}{2}(1-x)x^{v-1} \left\{ 2\log x+v (\log x)^2 \right\} \leq 0
$
for $x \geq 1$, we have only to prove $u_{3/4}(x) \geq 0$ for $x \geq 1$.
Since $u_{3/4}(x) = \frac{x^{5/4}-3x+4x^{3/4}-5x^{1/4}+3}{8x^{1/4}}$, we set the function $\mathfrak{v}(x) \equiv x^{5/4}-3x+4x^{3/4}-5x^{1/4}+3$. Some calculations show $\mathfrak{v}(x) \geq \mathfrak{v}(x) =0$ which implies $u_{3/4} (x) \geq 0$. Hence our claim follows.

In this case, the first inequality in \eqref{104}, can be considered as a refinement of the first inequality in \eqref{k1}.
\item[(ii)] Consider the function 
\[{{w}_{v }}\left( x \right)\equiv R{{\left( 1-\sqrt{x} \right)}^{2}}-v \left( x-1 \right)\left( 1-{{\left( \frac{x+1}{2} \right)}^{v -1}} \right).\] 
For $\frac{2}{3} \leq v \leq 1$, we have ${{w}_{v}}\left( x \right)\ge 0$. For proving this inequality, let $\mathfrak{x}_v(x) =(1-\sqrt{x})^2 -(x-1)\left(1-\left(\frac{x+1}{2}\right)^{v-1} \right)$. For $x \geq 1$, we then have
$\frac{d\mathfrak{x}_v(x)}{dv} = (x-1)\left( \frac{x+1}{2}\right)^{v-1} \left\{ \log \left( \frac{x+1}{2}\right)  \right\} \geq 0$.
We have only to prove $\mathfrak{x}_{2/3}(x) \geq 0$ for $x \geq 1$. By slightly complicated calculations, we have
$$
\mathfrak{x}_{2/3}(x) = \frac{2^{4/3} (\sqrt{x} -1)}{(x+1)^{1/3}} \left\{ \frac{\sqrt{x} +1}{2} -\left( \frac{x+1}{2}\right)^{1/3}\right\} \geq 0.
$$
Indeed, for $t \geq 1$, we have $(t-1)(t^2+3) \geq 0$ which is equivalent to
$(t+1)^3 \geq 4(t^2+1)$. Putting $t = \sqrt{x}$, we obtain $\frac{(\sqrt{x} +1)^3}{8} \geq \frac{x+1}{2}$ which shows $\frac{\sqrt{x}+1}{2} \geq \left(\frac{x+1}{2}\right)^{1/3}$. Thus our assertion follows.
\item[(iii)]
In addition, for $0 \leq v \leq \frac{1}{3}$, we have ${{w}_{v}}\left( x \right)\ge 0$. In fact, since  $v \left(\frac{x+1}{2} \right)^{v-1}$ is increasing for $v$, we estimate the first derivative of $w_v(x)$ as
\[\begin{aligned}
  & \frac{d{{w}_{v}}\left( x \right)}{dv}=-{{\left( \sqrt{x}-1 \right)}^{2}}-\left( x-1 \right)+\left( x-1 \right){{\left( \frac{x+1}{2} \right)}^{v-1}}\left( 1+v\log \left( \frac{x+1}{2} \right) \right) \\ 
 & \le -{{\left( \sqrt{x}-1 \right)}^{2}}-\left( x-1 \right)+\left( x-1 \right){{\left( \frac{x+1}{2} \right)}^{-{}^{2}/{}_{3}}}\left( 1+\frac{1}{3}\log \left( \frac{x+1}{2} \right) \right) \\ 
 & =-\frac{{{2}^{{}^{5}/{}_{3}}}\sqrt{x}\left( \sqrt{x}-1 \right)}{{{\left( x+1 \right)}^{{}^{2}/{}_{3}}}}\left\{ {{\left( \frac{x+1}{2} \right)}^{{}^{2}/{}_{3}}}-\frac{\sqrt{x}+1}{2\sqrt{x}}\left( 1+\log \left( \frac{x+1}{2} \right) \right) \right\} \\ 
 & \le 0. \\ 
\end{aligned}\]
The last inequality is due to Lemma \ref{lemma_0326}.
Consequently, $w_v(x) \geq w_{1/3}(x)$. So we prove $w_{1/3}(x) \geq 0$.
After short computations, we get
$$
w_{1/3}(x) = \frac{\sqrt{x} -1}{3}\left( \frac{x+1}{2}\right)^{-2/3} \left\{ (\sqrt{x} -3)\left( \frac{x+1}{2}\right)^{2/3} +\sqrt{x} +1 \right\}.
$$
Now we set the function $\mathfrak{y}(t) \equiv (t -3)\left( \frac{t^2+1}{2}\right)^{2/3} +t +1$ for $t \geq 1$.  By some calculations, we get $\mathfrak{y}(t) \geq \mathfrak{y}(1) =0$.
Therefore we have $w_v(t) \geq w_{1/3} (t) \geq 0$, as required.
\end{itemize}
In this cases, the second inequality in \eqref{104} provides an improvement for the second inequality in \eqref{k1} \footnote{It is interesting to note that, by the computer calculations, we find that if $v \ge 0.7$ then ${{u}_{v }}\left( x \right)\ge 0$ and if  $v \ge 0.6$ or $v \le 0.4$ we have ${{w}_{v }}\left( x \right)\ge 0$. These mean we have a possibility to extend the range of $v$ to satisfy the condition of (I-i), (I-ii) and (I-iii) in Proposition \ref{our_adv_prop}.}.
\item[(II)]  Let $x>0$.
It is clear that if ${{x}^{v}}\ge \frac{1}{2}$, then ${{M}_{v}}\left( x \right)\le K\left( x \right)$.
Indeed, by simple calculations, the inequality $M_v(x) \leq K(x)$ is equivalent to the inequality $2v(1-v) \leq x^v$. Since $v(1-v) \leq \frac{1}{4}$, we have $x^v \geq \frac{1}{2} \geq 2v(1-v)$ under the condition $x^v \geq \frac{1}{2}$.
\item[(III)] Dragomir obtained the inequality (\ref{ineq_Dragomir_comparison}) in \cite[Theorem 1]{dragomir} for $x>0$. However, for $0 \leq v \leq \frac{1}{2}$ and $0 < x \leq 1$, we show
\begin{equation} \label{ineq01_question0323}
M_v(x) \leq \exp \left( 4v(1-v) \left( K(x) -1 \right) \right).
\end{equation}
Our upper bound of Theorem \ref{b} is tighter than one given in {{\cite[Theorem 1]{dragomir}}}, when $0 \leq v \leq \frac{1}{2}$.

Let us prove the above inequality (\ref{ineq01_question0323}) which is equivalent to the inequality
$$
1+ \frac{1}{2x^v} \frac{v(1-v)(x-1)^2}{x} \leq \exp \left( \frac{v(1-v)(x-1)^2}{x} \right).
$$
We use the inequality
\[\exp \left( y \right)\ge 1+y+\frac{1}{2}{{y}^{2}},\qquad  y\ge 0, \]
with $y = \frac{v(1-v)(x-1)^2}{x} \geq 0$.
Then we calculate
\begin{eqnarray}
  && \exp \left( \frac{v\left( 1-v \right){{\left( x-1 \right)}^{2}}}{x} \right)-1-\frac{1}{2{{x}^{v}}}\frac{v\left( 1-v \right){{\left( x-1 \right)}^{2}}}{x} \nonumber \\ 
 && \ge \frac{v\left( 1-v \right){{\left( x-1 \right)}^{2}}}{x}\left( 1-\frac{1}{2{{x}^{v}}}+\frac{v\left( 1-v \right){{\left( x-1 \right)}^{2}}}{2x} \right) \nonumber \\ 
 && =\frac{v\left( 1-v \right){{\left( x-1 \right)}^{2}}}{x}\left( \frac{2{{x}^{v}}-1+v\left( 1-v \right){{x}^{v-1}}{{\left( x-1 \right)}^{2}}}{2{{x}^{v}}} \right). \label{ineq02_question0323}
\end{eqnarray}
Thus we have only to prove $2x^v -1 +v(1-v)x^{v-1}(x-1)^2 \geq 0$ for $0<x \leq 1$ and $0\leq v \leq \frac{1}{2}$. By putting $t = 1/x$, the above inequality becomes
$$
t^{-v-1} \left(2t-t^{v+1} +v(1-v)(t-1)^2 \right) \geq 0.
$$
Therefore it is sufficient to prove the inequality
$$
\mathfrak{g}_v(t) \equiv 2t-t^{v+1} +v(1-v)(t-1)^2 \geq 0,
$$
for $t \geq 1$  and $0 \leq v \leq \frac{1}{2}$. 
By some calculations, we have $\mathfrak{g}_v(t) \geq \mathfrak{g}_{1/2}(t) \geq \mathfrak{g}_{1/2}(1) =1>0$. Thus the proof of the inequality (\ref{ineq01_question0323}) was completed.\\
It should be mentioned here that the inequality (\ref{ineq01_question0323}) holds for $0 \leq v \leq 1$ and $x \geq \frac{1}{2}$ from (\ref{ineq02_question0323}).

\item[(IV)] It is natural to consider ${{m}_{v }}\left( x \right)$ and ${{M}_{v }}\left( x \right)$ are better than ${{K}^{r}}\left( x \right)$ and ${{K}^{R}}\left( x \right)$ under the assumption $0<x\le 1$.
\begin{itemize}
\item[(i)] In general, there is no ordering between ${{K}^{r}}\left( x \right)$ and ${{m}_{v }}\left( x \right)$. For this purpose, taking $v =0.3$ and $x=0.7$, then   
	\[{{m}_{v }}\left( x \right)-{{K}^{r}}\left( x \right)\approx 0.002.\] 
On the other hand, taking $v =0.7$ and $x=0.1$, we have
\[{{m}_{v }}\left( x \right)-{{K}^{r}}\left( x \right)\approx -0.15.\]
\item[(ii)] In addition, we have no ordering between ${{K}^{R}}\left( x \right)$ and ${{M}_{v }}\left( x \right)$. To see this putting $v =0.2$ and $x=0.4$, observe that 	
	\[{{K}^{R}}\left( x \right)-{{M}_{v }}\left( x \right)\approx 0.08.\] 
But if we choose $v =0.6$ and $x=0.3$ we get  
	\[{{K}^{R}}\left( x \right)-{{M}_{v }}\left( x \right)\approx -0.17.\] 
\end{itemize}
\end{itemize}
\hfill \qed

\section*{Appendix B}\label{app_A}
\vskip 0.4 true cm
\noindent{\it Proof of Proposition \ref{prop01_0330}.}
Since $g_{r,v}(t)$ is decreasing in $r$, $g_{r,v}(t) \geq g_{1,v}(t)$ so that we have only to prove for $t \geq 1$ and $0 \leq v \leq 1$, the inequality
$g_{1,v}(t) \geq \left\{ (1-v)+vt^{-1}\right\}^{-1}$
which is equivalent to the inequality by $v(t-1) \geq 0$
\begin{equation} \label{ineq02_0330}
\left(\frac{t+1}{2}\right)^{v-1} \geq \frac{1}{(1-v)t +v}.
\end{equation}
Since $t \geq 1$ and $0 \leq v \leq 1$, we have $t\left(\frac{t+1}{2}\right)^{v-1} \geq t^v$. In addition, for $t>0$ and $0\leq v \leq 1$, we have $t^v \geq \left\{ (1-v)+vt^{-1}\right\}^{-1}$. Thus we have $t\left(\frac{t+1}{2}\right)^{v-1} \geq \left\{ (1-v)+vt^{-1}\right\}^{-1}$ which implies the inequality (\ref{ineq02_0330}).
\hfill \qed

\medskip

\noindent{\it Proof of Proposition \ref{prop02_0330}.}
The first inequality is know for $t >0$ and $0 \leq v \leq 1$.
Since $g_{r,v}(t)$ is deceasing in $r$, in order to prove the second inequality we have only to prove $g_{1,v}(t) \geq t^v$, that is,
$$
v(t-1)\left( \frac{t+1}{2}\right)^{v-1} +1 \geq t^v.
$$
which is equivalent to the inequality
$$
\frac{t^v -1}{v} \leq (t-1) \left( \frac{t+1}{2}\right)^{v-1}.
$$
By the use of Hermite-Hadamard inequality with a convex function $x^{v-1}$ for $0\leq v \leq 1$ and $x >0$, the above inequality can be proven as
$$
\left( \frac{t+1}{2}\right)^{v-1} \leq \frac{1}{1-t} \int_t^1 x^{v-1} dx = \frac{1-t^v}{v(1-t)}. 
$$
\hfill \qed

\medskip

\noindent{\it Proof of Proposition \ref{prop05_0330}.}
We firstly prove $\mathsf{h}(t) \equiv 2(t-1)-\log t \geq 0$ for $c \leq t \leq 1$.
Since $\mathsf{h}''(t) \geq 0$, $\mathsf{h}(1)=0$ and $\mathsf{h}(c)  \approx -0.0000354367 <0.$
Thus we have $\mathsf{h}(t) \leq 0$ for $c \leq t \leq 1$.
Secondly we prove $\mathsf{l}_v(t) \equiv 2(t-1)-((1-v)t+v)\log t \leq 0$. Since $\frac{d\mathsf{l}_v(t)}{dv}=(t-1) \log t \geq 0$, we have $\mathsf{l}_v(t) \leq \mathsf{l}_1(t) =\mathsf{h}(t) \leq 0$.
Since $G_{r,v}(t)$ is decreasing in $r$, we have $G_{r,v}(t) \geq G_{1,v}(t)$ so that we have only to prove $G_{1,v}(t) \geq \left\{ (1-v)+vt^{-1}\right\}^{-1}$, which is equivalent to the inequality, by $v(t-1) \leq 0$
$$
 \frac{t^{v-1} +1}{2} \leq \frac{1}{(1-v)t+v},
$$
for $0 \leq r,v \leq 1$ and $c \leq t \leq 1$.
To this end, we set $\mathsf{f}_v(t) \equiv 2-( t^{v-1} +1)((1-v)t+v)$.
Some calculations imply $\mathsf{f}_v(t) \geq \mathsf{f}_1(t) =0$.
\hfill \qed

\medskip

\noindent{\it Proof of Proposition \ref{prop06_0330}.}
The first inequality is know for $t >0$ and $0 \leq v \leq 1$.
Since $G_{r,v}(t)$ is deceasing in $r$, in order to prove the second inequality we have only to prove $G_{1,v}(t) \geq t^v$,
which is equivalent to the inequality
$$
\frac{1}{2}v(t-1) \left(t^{v-1} +1 \right) +1 \geq t^v.
$$
To this end, we set
$$
\mathsf{k}_v(t) \equiv v(t-1)(t^{v-1} +1)+2-2t^v .
$$
Some calculations imply $\mathsf{k}_v(t) \geq \mathsf{k}_v(1) =0$.
\hfill \qed

\vskip 1.4 true cm

{\tiny $^1$Department of Information Science, College of Humanities and Sciences, Nihon University, 3-25-40, Sakurajyousui, Setagaya-ku, Tokyo, 156-8550, Japan.

{\it E-mail address:} furuichi@chs.nihon-u.ac.jp

\vskip 0.4 true cm

$^2$Young Researchers and Elite Club, Mashhad Branch, Islamic Azad University, Mashhad, Iran.

{\it E-mail address:} hrmoradi@mshdiau.ac.ir}
\end{document}